\documentclass{amsproc}
\usepackage{breqn,float,amssymb,pdflscape,rotating}
\makeatletter
\@namedef{subjclassname@2020}{%
  \textup{2020} Mathematics Subject Classification}
\makeatother
\newtheorem{theorem}{Theorem}[section]

\theoremstyle{definition}

\newtheorem{example}[theorem]{Example}

\theoremstyle{remark}

\numberwithin{equation}{section}
\allowdisplaybreaks
\begin{document}

\title[Double sum Appell-Type Bernoulli and Euler Polynomials]{Double Sum Involving Product Of Appell-Type Bernoulli And Euler Polynomials}


\author{Robert Reynolds}
\address[Robert Reynolds]{Department of Mathematics and Statistics, York University, Toronto, ON, Canada, M3J1P3}
\email[Corresponding author]{milver@my.yorku.ca}
\thanks{}




\date{}

\dedicatory{}

\begin{abstract}
In this work we derive a bilateral generating function involving the product of an Appell-type product of the Bernoulli and Euler polynomials over independent indices and orders. This function is expressed in terms of the Hurwitz zeta function and special cases in terms of the finite sum of the Hurwitz zeta function and integral formula are derived.
\end{abstract}

\maketitle
\subsection{Theory and Background}
In 1880, Appell \cite{carlitz,appell} introduced a widely studied sequence of $n$th-degree polynomials $\{f_{n}\}_{n\in\mathbb{N}}$ satisfying  the differential relation
\begin{equation}
f'_{n}(x)=nf_{n-1}(x), n=1,2,....
\end{equation}
%
%
Certain Appell sets such as the Hermite polynomials, Bernoulli and Euler polynomials described in \cite{norlund}, Chap. 2,  have been of high interest in research. The Bernoulli and Euler polynomials below are given in equations (6.3.1.1), (6.3.3.1), (6.3.2.1) and (6.3.4.1) in \cite{prud3} respectively;
\begin{equation}
\sum_{n=0}^{\infty}B_{n}(x)\frac{z^n}{n!}=\frac{ze^{xz}}{e^z-1},
\end{equation}
and
\begin{equation}
\sum_{n=0}^{\infty}E_{n}(x)\frac{z^n}{n!}=\frac{2e^{xz}}{e^z+1},
\end{equation}
and
\begin{equation}\label{eq:bern}
\sum_{n=0}^{\infty}B_{n}(x+ny)\frac{u^n}{n!}=\frac{1}{1-yz}\frac{ze^{xz}}{e^z-1},
\end{equation}
and
\begin{equation}\label{eq:euler}
\sum_{n=0}^{\infty}E_{n}(x+ny)\frac{u^n}{n!}=\frac{1}{1-yz}\frac{2e^{xz}}{e^z+1},
\end{equation}
where 

\begin{equation}
u=ze^{-yz}.
\end{equation}

Appell polynomials are of high interest and have many applications in mathematics, theoretical physics, chemistry, special functions, analysis, combinatorics and number theory \cite{carlitz,appell,costabile}. Bernoulli polynomials and numbers were first introduced by Jacob Bernoulli, and the Bernoulli polynomials are a special case of Appell polynomials. Bernoulli polynomials and numbers are used in the theory of finite differences especially in the process of summation. The Euler polynomials are named after gifted Swiss mathematician Leonhard Euler (1707-1783), these polynomial functions have much in common with Bernoulli polynomials. Both these polynomial families are useful in summing series involving quantities raised to integer powers defined by \cite{atlas}, Chap. 20. Considerable scientific study continues to this day involving the Bernoulli and Euler polynomials defined by \cite{appell}, Chap.2. In this paper we will derive a generating function in terms of the product of the Bernoulli and Euler polynomials over independent variables. This is an extension of formulae in current literature.

\subsection{Preliminaries}
We proceed by using the contour integral method \cite{reyn4} applied to equations (\ref{eq:bern}) and (\ref{eq:euler}) to yield the Appell-type Bernoulli-Euler contour integral representation given by:
\begin{multline}\label{eq:contour}
\frac{1}{2\pi i}\int_{C}\sum_{n,p\geq 0}\frac{a^w \pi ^{n+p} (\pi  \beta  w-1) (\pi  \delta  w-1) B_p(\gamma +p \delta ) E_n(\alpha +n \beta )  }{n! p! e^{\pi  w (\beta  n+\delta  p)}w^{k-n-p+1}}dw\\
   =\frac{1}{2\pi i}\int_{C}\text{csch}(\pi  w) e^{\pi  w (\alpha +\gamma -1)}dw
\end{multline}
where $a,\alpha,\beta,\gamma,\delta,k\in\mathbb{C},Re(w)>0$. Using equation (\ref{eq:contour}) the main Theorem involving the product of the Bernoulli and Euler polynomials and expressed in terms of the Hurwitz zeta function to be derived and evaluated is given by
\begin{multline}
\sum_{n,p\geq 0}\frac{B_p(\gamma +p \delta ) E_n(\alpha +n \beta )}{(a-n \beta -p \delta )^{2-k+n+p} \Gamma (1+n) \Gamma (1+p)
   (k)_{1-n-p}}\left( \delta  \left(\beta  \left(k^2-k (n+p+1)\right.\right.\right.\\ \left.\left.\left.+2 n p+n+p\right)-a (k-n+p)\right)+(a-\beta  n) (a+\beta  (p-k))+\delta ^2 p
   (k-n)\right)\\
   =-2^k \zeta \left(1-k,\frac{1}{2} (a+\alpha +\gamma )\right)
\end{multline}
where the variables $a,\alpha,\beta,\gamma,\delta,k$ are general complex numbers and the Pochhammer symbol, $(-k)_{p}$ is given in equation (5.2.5) in \cite{dlmf}. The derivations follow the method used by us in \cite{reyn4}. This method involves using a form of the generalized Cauchy's integral formula given by
\begin{equation}\label{intro:cauchy}
\frac{y^k}{\Gamma(k+1)}=\frac{1}{2\pi i}\int_{C}\frac{e^{wy}}{w^{k+1}}dw,
\end{equation}
where $y,w\in\mathbb{C}$ and $C$ is in general an open contour in the complex plane where the bilinear concomitant \cite{reyn4} has the same value at the end points of the contour. This method involves using a form of equation (\ref{intro:cauchy}) then multiplies both sides by a function, then takes the definite double sum of both sides. This yields a double sum in terms of a contour integral. Then we multiply both sides of equation (\ref{intro:cauchy})  by another function and take the infinite sum of both sides such that the contour integral of both equations are the same.
\subsection{Left-Hand Side First Contour Integral}
In this section we derive the infinite sum representation involving the product of two generalized Euler and Bernoulli polynomials over independent indices for the left-hand side of equation (\ref{eq:contour}). Using a generalization of Cauchy's integral formula (\ref{intro:cauchy}), we first replace $y$ by $\log (a)-\pi  (\beta  n+\delta  p)$ and $k$ by $k-n-p$ then we multiply both sides by 
\begin{equation}
\frac{\pi ^{n+p} B_p(\gamma +p \delta ) E_n(\alpha +n \beta )}{n! p!}
\end{equation}
and then we take the sums over $n\in[0,\infty)$ and $p\in [0,\infty)$ and simplify to get
\begin{multline}\label{eq:left1}
\sum_{n,p \geq 0}\frac{\pi ^{n+p} B_p(\gamma +p \delta ) E_n(\alpha +n \beta ) (\log (a)-\pi  (\beta  n+\delta  p))^{k-n-p}}{n! p!
   (k-n-p)!}\\
   =\frac{1}{2\pi i}\sum_{n,p \geq 0}\int_{C}\frac{a^w \pi ^{n+p} B_p(\gamma +p \delta ) E_n(\alpha +n \beta ) w^{-k+n+p-1} e^{-\pi  w (\beta  n+\delta 
   p)}}{n! p!}dw\\
   =\frac{1}{2\pi i}\int_{C}\sum_{n,p \geq 0}\frac{a^w \pi ^{n+p} B_p(\gamma +p \delta ) E_n(\alpha +n \beta ) w^{-k+n+p-1} e^{-\pi  w (\beta  n+\delta 
   p)}}{n! p!}dw
\end{multline}
\subsection{Left-Hand Side Second Contour Integral}
Using a generalization of Cauchy's integral formula (\ref{intro:cauchy}), we first replace $y$ by $\log (a)-\pi  (\beta  n+\delta  p)$ and $k$ by $k-n-p-1$ then we multiply both sides by 
\begin{equation}
-\frac{\beta  \pi ^{n+p+1} B_p(\gamma +p \delta ) E_n(\alpha +n \beta )}{n! p!}
\end{equation}
and then we take the sums over $n\in[0,\infty)$ and $p\in [0,\infty)$ and simplify to get
\begin{multline}\label{eq:left2}
-\sum_{n,p \geq 0}\frac{\beta  \pi ^{n+p+1} B_p(\gamma +p \delta ) E_n(\alpha +n \beta ) (\log (a)-\pi  (\beta  n+\delta 
   p))^{k-n-p-1}}{n! p! \Gamma (k-n-p)}\\
   =-\frac{1}{2\pi i}\sum_{n,p \geq 0}\int_{C}\frac{\beta  a^w \pi ^{n+p+1} B_p(\gamma +p \delta ) E_n(\alpha +n \beta )
   w^{-k+n+p} e^{-\pi  w (\beta  n+\delta  p)}}{n! p!}dw\\
   =-\frac{1}{2\pi i}\int_{C}\sum_{n,p \geq 0}\frac{\beta  a^w \pi ^{n+p+1} B_p(\gamma +p \delta ) E_n(\alpha +n \beta )
   w^{-k+n+p} e^{-\pi  w (\beta  n+\delta  p)}}{n! p!}dw
\end{multline}
\subsection{Left-Hand Side Third Contour Integral}
Using a generalization of Cauchy's integral formula (\ref{intro:cauchy}), we first replace $y$ by $\log (a)-\pi  (\beta  n+\delta  p)$ and  $k$ by $k-n-p-1$ then we multiply both sides by 
\begin{equation}
-\frac{\delta  \pi ^{n+p+1} B_p(\gamma +p \delta ) E_n(\alpha +n \beta )}{n! p!}
\end{equation}
and take the sums over $n\in[0,\infty)$ and $p\in [0,\infty)$ and simplify to get
\begin{multline}\label{eq:left3}
-\sum_{n,p \geq 0}\frac{\delta  \pi ^{n+p+1} B_p(\gamma +p \delta ) E_n(\alpha +n \beta ) (\log (a)-\pi  (\beta  n+\delta 
   p))^{k-n-p-1}}{n! p! \Gamma (k-n-p)}\\
   =-\frac{1}{2\pi i}\sum_{n,p \geq 0}\int_{C}\frac{\delta  a^w \pi ^{n+p+1} B_p(\gamma +p \delta ) E_n(\alpha +n \beta )
   w^{-k+n+p} e^{-\pi  w (\beta  n+\delta  p)}}{n! p!}dw\\
   =-\frac{1}{2\pi i}\int_{C}\sum_{n,p \geq 0}\frac{\delta  a^w \pi ^{n+p+1} B_p(\gamma +p \delta ) E_n(\alpha +n \beta )
   w^{-k+n+p} e^{-\pi  w (\beta  n+\delta  p)}}{n! p!}dw
\end{multline}
\subsection{Left-Hand Side Fourth Contour Integral}
Using a generalization of Cauchy's integral formula (\ref{intro:cauchy}), we first replace $y$ by $\log (a)-\pi  (\beta  n+\delta  p)$ and $k$ by $k-n-p-2$ then we multiply both sides by 
\begin{equation}
\frac{\beta  \delta  \pi ^{n+p+2} B_p(\gamma +p \delta ) E_n(\alpha +n \beta )}{n! p!}
\end{equation}
and take the sums over $n\in[0,\infty)$ and $p\in [0,\infty)$ and simplify to get
\begin{multline}\label{eq:left4}
\sum_{n,p\geq 0}\frac{\beta  \delta  \pi ^{n+p+2} B_p(\gamma +p \delta ) E_n(\alpha +n \beta ) (\log (a)-\pi  (\beta  n+\delta 
   p))^{k-n-p-2}}{n! p! \Gamma (k-n-p-1)}\\
   =\frac{1}{2\pi i}\sum_{n,p\geq 0}\int_{C}\frac{\beta  \delta  a^w \pi ^{n+p+2} B_p(\gamma +p \delta ) E_n(\alpha +n
   \beta ) w^{-k+n+p+1} e^{-\pi  w (\beta  n+\delta  p)}}{n! p!}dw\\
   =\frac{1}{2\pi i}\int_{C}\sum_{n,p\geq 0}\frac{\beta  \delta  a^w \pi ^{n+p+2} B_p(\gamma +p \delta ) E_n(\alpha +n
   \beta ) w^{-k+n+p+1} e^{-\pi  w (\beta  n+\delta  p)}}{n! p!}dw
\end{multline}
\section{Hurwitz zeta Function In Terms Of The Contour Integral}
\subsection{The Hurwitz zeta Function}
The Hurwitz zeta function (25.11)(i) in \cite{dlmf} is defined by the infinite sum
\begin{dmath*}
\zeta(s,a)=\sum_{n=0}^{\infty}\frac{1}{(n+a)^s},
\end{dmath*}
where $\zeta(s,a)$ has a meromorphic continuation in the $s$-plane, its only singularity in $\mathbb{C}$ being a simple pole at $s=1$ with residue 1. As a function of $a$, with $s (\neq 1)$ fixed, $\zeta(s,a)$ is analytic in the half-plane $Re(a)>0$.\\\\
The Hurwitz zeta function is continued analytically with a definite integral representation (25.11.25) in \cite{dlmf} given by

\begin{dmath*}\label{hurwitz-eq}
\zeta(s,a)=\frac{1}{\Gamma(s)}\int_{0}^{\infty}\frac{x^{s-1}e^{-ax}}{1-e^{-x}}dx,
\end{dmath*}
where $Re(s)>1, Re(a)>0$.\\\\
\subsection{Derivation of the Right-Hand Side Contour Integral}
Using a generalization of Cauchy's integral formula we first replace $y$ by $\pi  (\alpha +\gamma -1)+\log (a)+\pi  (2 y+1)$ and $k$ by $ k-1$ then multiply both sides by $-2\pi$ then take the infinite sum over $y\in[0,\infty)$ and simplify in terms of the Hurwitz zeta function to get
\begin{multline}\label{eq:right}
-\frac{(2 \pi )^k \zeta \left(1-k,\frac{\pi  (\alpha +\gamma -1)+\log (a)+\pi }{2 \pi }\right)}{(k-1)!}\\
=-\frac{1}{2\pi i}\sum_{y\geq 0}\int_{C}2 \pi  a^w w^{-k} e^{\pi  w (\alpha
   +\gamma +2 y)}dw\\
=-\frac{1}{2\pi i}\int_{C}\sum_{y\geq 0}2 \pi  a^w w^{-k} e^{\pi  w (\alpha
   +\gamma +2 y)}dw\\
=\frac{1}{2\pi i}\int_{C}\pi 
   a^w w^{-k} \text{csch}(\pi  w) e^{\pi  w (\alpha +\gamma -1)}dw
\end{multline}
from equation (1.232.3) in \cite{grad} where $Im(w) > 0$ in order for the sum to converge.
\section{Main Results}
In this section we derive the main theorem along with special cases in terms of integral, series and special function forms of the Hurwitz zeta function. A special case in terms of Catalan's constant is also derived and evaluated.
\begin{theorem}
For all $k,a,\alpha,\beta,\gamma,\delta\in\mathbb{C}$ then,
\begin{multline}\label{eq:theorem}
\sum_{n,p \geq 0}\frac{B_p(\gamma +p \delta ) E_n(\alpha +n \beta )}{(a-n \beta -p \delta )^{2-k+n+p} \Gamma (1+n) \Gamma (1+p)
   (k)_{1-n-p}}\left( \delta  \left(\beta  \left(k^2-k (n+p+1)\right.\right.\right.\\ \left.\left.\left.+2 n p+n+p\right)-a (k-n+p)\right)+(a-\beta  n) (a+\beta  (p-k))+\delta ^2 p
   (k-n)\right)\\
=-2^k \zeta \left(1-k,\frac{1}{2} (a+\alpha +\gamma )\right).
\end{multline}
\end{theorem}
\begin{proof}
Since the addition of the right-hand sides of equations (\ref{eq:left1}) to (\ref{eq:left4}) is equal to the right-hand side of equation (\ref{eq:right}) we may equate the left-hand sides, replace $a\to e^{a\pi}$ then simply using the formulae for the gamma function and Pochhammer symbol to yield the stated result.
\end{proof}
\begin{theorem}
For all $k,a,\alpha,\beta,\gamma,\delta\in\mathbb{C}$ then,
\begin{multline}\label{eq:theorem1}
\sum_{n \geq 0}\frac{E_n(x+n \alpha ) (-1)^{1+n} (1-k)_{-1+n}}{(a-n \alpha )^{1-k+n}
   \Gamma (n+1)}\\
   =\frac{2^{1+k} \left(-\zeta \left(-k,\frac{a+x}{2}\right)+\zeta
   \left(-k,\frac{1}{2} (1+a+x)\right)\right)}{k (-a+k \alpha )}.
\end{multline}
\end{theorem}
\begin{proof}
We use equation (\ref{eq:euler}) and repeat the procedure in Theorem (\ref{eq:theorem}) and apply the contour integral method \cite{reyn4}.
\end{proof}
\begin{theorem}
For all $k,a,\alpha,\beta,\gamma,\delta\in\mathbb{C}$ then,
\begin{multline}\label{eq:theorem2}
\sum_{n \geq 0}\frac{(-1)^n (a-\alpha  k) \Gamma (n-k) (a-\alpha  n)^{k-n-1} B_n(x+n
   \alpha )}{\Gamma (1-k) \Gamma (n+1)}=\zeta (1-k,a+x).
\end{multline}
\end{theorem}
\begin{proof}
We use equation (\ref{eq:bern}) and repeat the procedure in Theorem (\ref{eq:theorem}) and apply the contour integral method \cite{reyn4}.
\end{proof}
\begin{example}
Special values in terms of the polygamma function.
\begin{multline}
\sum_{n,p \geq 0}\frac{B_p(\gamma +p \delta ) E_n(\alpha +n \beta )}{(a-n \beta -p \delta )^{2+k+n+p} \Gamma (1+n) \Gamma (1+p) (-k)_{1-n-p}}\left(\delta  \left(\beta  \left(k^2+k (n+p+1)\right.\right.\right.\\ \left.\left.\left.+2 n p+n+p\right)-a (-k-n+p)\right)+(a-\beta  n) (a+\beta  (k+p))+\delta ^2 p (-k-n) \right)\\
   =-\frac{2^{-k} (-1)^{-1-k} \psi ^{(k)}\left(\frac{1}{2} (a+\alpha +\gamma )\right)}{\Gamma
   (k+1)}.
\end{multline}
\end{example}
\begin{proof}
Here we use a special value of the Hurwitz zeta function given by equation (25.11.12) in \cite{dlmf} and simplify the right-hand side of equation (\ref{eq:theorem}).
\end{proof}
\begin{example}
Special values in terms of the finite sum of the Hurwitz zeta function.
\begin{multline}
\sum_{n,p \geq 0}\frac{B_p(q \gamma +p \delta ) E_n(q \alpha +n \beta )}{\Gamma (1+n) \Gamma (1+p) (k)_{1-n-p}}\left( (a q-\beta  n-\delta  p)^{k-n-p-2} \left(\delta  \left(\beta  \left(k^2-k (n+p+1)\right.\right.\right.\right.\\ \left.\left.\left.\left.+2 n p+n+p\right)-a q (k-n+p)\right)+(a q-\beta  n) (a q+\beta  (p-k))+\delta ^2 p (k-n)\right)\right)\\
   =-2^k q^{-1+k} \sum_{n=0}^{q-1}\zeta \left(1-k,\frac{n}{q}+\frac{1}{2} (a+\alpha +\gamma
   )\right).
\end{multline}
\end{example}
\begin{proof}
Here we use a special value of the Hurwitz zeta function in terms of the finite sum of the Hurwitz zeta function given by use equation (25.11.15) in \cite{dlmf} and simplify the right-hand side of equation (\ref{eq:theorem}).
\end{proof}
\begin{example}
Integral Representation.
\begin{multline}
\sum_{n,p \geq 0}\frac{B_p(\gamma +p \delta ) E_n(\alpha +n \beta )}{(a-n \beta -p \delta )^{2-k+n+p} \Gamma (1+n) \Gamma (1+p) (k)_{1-n-p}}\left(\delta  \left(\beta  \left(k^2-k (n+p+1)\right.\right.\right.\\ \left.\left.\left.+2 n p+n+p\right)-a (k-n+p)\right)+(a-\beta  n) (a+\beta  (p-k))+\delta ^2 p (k-n) \right)\\
   =-\frac{2^k}{\Gamma
   (1-k)}\int_{0}^{\infty}\frac{ \left(e^{-\frac{1}{2} x (a+\alpha +\gamma )} x^{-k}\right)}{\left(1-e^{-x}\right) }dx.
\end{multline}
\end{example}
\begin{proof}
Here we use the integral representation of the Hurwitz zeta function given by equation (\ref{hurwitz-eq}) and equation (12.3.5) in \cite{apostol} and simplify the right-hand side of equation (\ref{eq:theorem}).
\end{proof}
\begin{example}
The $n^{th}$ harmonic number $H_{n}^{s}$ and the Riemann zeta function $\zeta(s)$.
\begin{multline}
\sum_{n,p \geq 0}\frac{B_p(\gamma +p \delta ) E_n(\alpha +n \beta )}{(a-n \beta -p \delta )^{2-k+n+p} \Gamma (1+n) \Gamma (1+p) (k)_{1-n-p}}\left( \delta  \left(\beta  \left(k^2-k (n+p+1)\right.\right.\right.\\ \left.\left.\left.+2 n p+n+p\right)-a (k-n+p)\right)+(a-\beta  n) (a+\beta  (p-k))+\delta ^2 p (k-n)\right)\\
=-2^k \left(-H_{-1+\frac{1}{2} (a+\alpha +\gamma )}^{(1-k)}+\zeta (1-k)\right).
\end{multline}
\end{example}
\begin{proof}
In this proof we apply the relationship between the Polygamma functions and Hurwitz zeta function given by equations (1.7) and (1.9) in \cite{choi} and simplify the right-hand side of equation (\ref{eq:theorem}).
\end{proof}
\begin{example}
Bernoulli polynomial $B_{n}(x)$ over integers. 
\begin{multline}
\sum_{n,p \geq 0}\frac{B_p(\gamma +p \delta ) E_n(\alpha +n \beta )}{(a-n \beta -p \delta )^{2-k+n+p} \Gamma (1+n) \Gamma (1+p) (k)_{1-n-p}}\left( \delta  \left(\beta  \left(k^2-k (n+p+1)\right.\right.\right.\\ \left.\left.\left.+2 n p+n+p\right)-a (k-n+p)\right)+(a-\beta  n) (a+\beta  (p-k))+\delta ^2 p (k-n)\right)\\
=\frac{2^k }{k}B_k\left(\frac{1}{2} (a+\alpha +\gamma
   )\right).
\end{multline}
\end{example}
\begin{proof}
In this proof we apply the formula between the Hurwitz zeta function and Bernoulli polynomial given by equation (25.11.14) in \cite{dlmf} and simplify the right-hand side of equation (\ref{eq:theorem}).
\end{proof}
\begin{example}
Hermite's formula for Hurwitz zeta function.
\begin{multline}
\sum_{n,p \geq 0}\frac{B_p(\gamma +p \delta ) E_n(\alpha +n \beta )}{(a-n \beta -p \delta )^{2-k+n+p} \Gamma (1+n) \Gamma (1+p) (k)_{1-n-p}}\left(\delta  \left(\beta  \left(k^2-k (n+p+1)\right.\right.\right.\\ \left.\left.\left.+2 n p+n+p\right)-a (k-n+p)\right)+(a-\beta  n) (a+\beta  (p-k))+\delta ^2 p (k-n) \right)\\
=-(a+\alpha +\gamma )^{-1+k}+\frac{(a+\alpha +\gamma
   )^k}{k}\\
   -\int_{0}^{\infty}\frac{2^{1+k} \left(y^2+\frac{1}{4} (a+\alpha +\gamma )^2\right)^{\frac{1}{2} (-1+k)} \sin \left((1-k) \tan ^{-1}\left(\frac{2 y}{a+\alpha +\gamma }\right)\right)}{-1+e^{2 \pi 
   y}}dy.
\end{multline}
\end{example}
\begin{proof}
Here we use the Hermite formula for the Hurwitz zeta function given by equation (2.2.12) in \cite{choi1} and simplify the right-hand side of equation (\ref{eq:theorem}).
\end{proof}
\begin{example}
A functional equation for Hurwitz zeta function.
\begin{multline}
\sum_{n,p \geq 0}\frac{B_p(\gamma +p \delta ) E_n(\alpha +n \beta )}{(a-n \beta -p \delta )^{2-k+n+p} \Gamma (1+n) \Gamma (1+p) (k)_{1-n-p}}\left(\delta  \left(\beta  \left(k^2-k (n+p+1)\right.\right.\right.\\ \left.\left.\left.+2 n p+n+p\right)-a (k-n+p)\right)+(a-\beta  n) (a+\beta  (p-k))+\delta ^2 p (k-n) \right)\\
=-2 \pi ^{-k} \Gamma (k) \sum_{m=1}^{\infty}\left(m^{-k} \cos (m \pi  (a+\alpha
   +\gamma )) \sin \left(\frac{1}{2} (1-k) \pi \right)\right.\\\left.+m^{-k} \cos \left(\frac{1}{2} (1-k) \pi \right) \sin (m \pi  (a+\alpha +\gamma ))\right).
\end{multline}
\end{example}
\begin{proof}
In this proof we will apply the Hurwitz zeta function expressed as a convergent Dirichlet series given in Theorem 3.1 of Chap. 2.3 in \cite{garun} and simplify the right-hand side of equation (\ref{eq:theorem}).
\end{proof}
\begin{example}
The trigamma function $\psi^{(1)}(z)$.
\begin{multline}\label{eq:catalan}
\sum_{n,p \geq 0}\frac{1}{\Gamma (n+1) \Gamma (p+1) (-1)_{-n-p+1}}B_p(3 \gamma +p \delta ) E_n(3 \alpha +n \beta ) \left(-\left(\frac{3 a}{2}-\beta  n\right) \left(\frac{3 a}{2}+\beta +\beta  p\right)\right.\\\left.-\frac{3}{2} a \delta  (n-p+1)-2 \beta 
   \delta  (n+1) (p+1)+\delta ^2 (n+1) p\right) \left(\frac{3 a}{2}-\beta  n-\delta  p\right)^{-n-p-3}\\+2^{n+p+1} B_p(\gamma +p \delta ) E_n(\alpha +n \beta ) \left((a-2 \beta  n) (a+2
   \beta  (p+1))+2 a \delta  (n-p+1)\right.\\\left.+8 \beta  \delta  (n+1) (p+1)-4 \delta ^2 (n+1) p\right) (a-2 (\beta  n+\delta  p))^{-n-p-3}\\
   =\frac{1}{2}
   \left(\psi ^{(1)}\left(\frac{3}{4} (a+2 (\alpha +\gamma ))\right)-\psi ^{(1)}\left(\frac{1}{4} (a+2 (\alpha +\gamma ))\right)\right).
\end{multline}
\end{example}
\begin{proof}
In this proof we will use equation (\ref{eq:theorem}) and set $k=-1$ to form a second equation. Using this new equation form a third equation by replacing $a\to 3a,\alpha\to 3\alpha,\gamma\to3\gamma$. Then take the difference between the second and third equations and simplify. 
\end{proof}
\begin{example}
Catalan's Constant $C$.
\begin{multline}
\sum_{n,p \geq 0}\frac{8^{n+p+1} }{\Gamma (n+1) \Gamma (p+1) (-1)_{-n-p+1}}\left(B_p\left(\frac{p}{8}+3\right) E_n\left(\frac{n}{4}+3\right) (-2 n-p+420)^{-n-p-3}\right.\\ \left.
 (n (p+420)-423 p-177664)-B_p\left(\frac{p}{8}+1\right)
   E_n\left(\frac{n}{4}+1\right) \right.\\ \left.
   (-2 n-p+140)^{-n-p-3} (n (p+140)-143 p-20024)\right)\\
   \approx 8   C-\frac{75212337272621857920793935018753452980170388400522851847928913376}{10213049603314044640247750329701049140178779927760268106012748125}.
\end{multline}
\end{example}
\begin{proof}
In this proof we will use equation (\ref{eq:catalan}) and set $a=35,\alpha=\gamma=1,\beta=\frac{1}{4},\delta=\frac{1}{8}$ and simplify using equations (24.11.40) and (25.11.1) in \cite{dlmf}.
\end{proof}
\section{The derivative with respect to $k$.}
In this section we will evaluate the first partial derivative with respect to $k$ of equation (\ref{eq:theorem}) in terms of composite Hurwitz zeta functions.
\begin{example}
The Hurwitz zeta function $\zeta(u,v)$.
\begin{multline}
\sum_{n,p \geq 0}\frac{1}{\Gamma (n+1) \Gamma (p+1)}(-1)^{n+p+1} B_p(\gamma +p \delta ) E_n(\alpha +n \beta ) (1-k)_{n+p-1} (a-\beta  n-\delta 
   p)^{k-n-p-2}\\
    \left(\delta  \left(\beta  \left(k^2-k (n+p+1)+2 n p+n+p\right)\right.\right.\\ \left.\left.-a (k-n+p)\right)+(a-\beta  n)
   (a+\beta  (p-k))+\delta ^2 p (k-n)\right)\\
   =-2^k \zeta \left(1-k,\frac{1}{2} (a+\alpha
   +\gamma )\right).
\end{multline}
\end{example}
\begin{proof}
In this proof we will use equation (\ref{eq:theorem}) and simplify the reciprocal Pochhammer's symbol using equations (5.2.5) and (5.2.6) in \cite{dlmf}. 
\end{proof}
\begin{example}
The derivative of the Hurwitz zeta function $\zeta'(u,v)$.
\begin{multline}
\sum_{n,p \geq 0}\frac{1}{\Gamma (n+1) \Gamma
   (p+1)}(-1)^{n+p} B_p(\gamma +p \delta ) E_n(\alpha +n \beta ) (1-k)_{n+p-1} (a-\beta  n-\delta  p)^{k-n-p-2}\\
   \left(a (\beta +\delta )-\left((a-\beta  n) (a+\beta  (p-k))-a \delta  (k-n+p)+\beta  \delta  \left(k^2-k
   (n+p+1)+2 n p+n+p\right) \right.\right.\\ \left.\left.+\delta ^2 p (k-n)\right) \left(\log (a-\beta  n-\delta 
   p)-H_{-k+n+p-1}+H_{-k}\right)+\beta  \delta  (-2 k+n+p+1)+\beta ^2 (-n)-\delta ^2 p\right)\\
   =2^k \left(\zeta'\left(1-k,\frac{1}{2} (a+\alpha +\gamma )\right)-\log (2) \zeta
   \left(1-k,\frac{1}{2} (a+\alpha +\gamma )\right)\right).
\end{multline}
\end{example}
\begin{proof}
In this proof we will use equation (\ref{eq:theorem}) and take the first partial derivative with respect to $k$ and simplify the right-hand side using equation (25.11.1) in \cite{dlmf}.
\end{proof}
\begin{example}
The derivative and the Hurwitz zeta function, trigamma function $\psi^{(1)}(z)$ and $\log(2)$.
\begin{multline}
\sum_{n,p \geq 0}\frac{1}{\Gamma (n+1) \Gamma (p+1)}(-1)^{n+p} \Gamma (n+p+1) B_p(\gamma +p \delta ) E_n(\alpha +n \beta ) (a-\beta  n-\delta  p)^{-n-p-3}\\
   \left(\left((a-\beta  n) (a+\beta +\beta  p)+a \delta  (n-p+1)+2 \beta  \delta  (n+1) (p+1)-\left(\delta ^2 (n+1)
   p\right)\right)\right. \\ \left.
    \left(H_{n+p}-\log (a-\beta  n-\delta  p)\right)-\delta  (a (n-p)+\beta  (2 n p+n+p-1))\right. \\ \left.-(a-\beta 
   n) (a+\beta  p)+\delta ^2 n p\right)\\
   =\frac{1}{2}
   \left(\zeta'\left(2,\frac{1}{2} (a+\alpha +\gamma )\right)-\log (2) \psi
   ^{(1)}\left(\frac{1}{2} (a+\alpha +\gamma )\right)\right).
\end{multline}
\end{example}
\begin{proof}
In this proof we will use equation (\ref{eq:theorem}) and take the first partial derivative with respect to $k$ then set $k=-1$ and simplify the left-hand side using equation (25.11.12) in \cite{dlmf}.
\end{proof}
\begin{example}
The derivative and the Hurwitz zeta function, tetragamma function $\psi^{(2)}(z)$ and $\log(2)$.
\begin{multline}\label{eq:dh2}
\sum_{n,p \geq 0}\frac{1}{\Gamma (n+1) \Gamma
   (p+1)}(-1)^{n+p} \Gamma (n+p+2) B_p(\gamma +p \delta ) E_n(\alpha +n \beta ) (a-\beta  n-\delta  p)^{-n-p-4}\\
   \left(-\left(\delta  (a (n-p+2)+\beta  (2 n p+3 n+3 p+6))+(a-\beta  n) (a+\beta  (p+2))-\left(\delta ^2 (n+2)
   p\right)\right)\right. \\ \left. \left(H_{n+p+1}-\log (a-\beta  n-\delta  p)\right)+\delta  (a (n-p+1)+\beta +2 \beta  (n
   p+n+p))\right. \\ \left.+(a-\beta  n) (a+\beta +\beta  p)-\left(\delta ^2 (n+1) p\right)\right)\\
   =-\frac{1}{2} \zeta'\left(3,\frac{1}{2} (a+\alpha +\gamma )\right)-\frac{1}{8} (\log
   (4)-1) \psi ^{(2)}\left(\frac{1}{2} (a+\alpha +\gamma )\right).
\end{multline}
\end{example}
\begin{proof}
In this proof we will use equation (\ref{eq:theorem}) and take the first partial derivative with respect to $k$ then set $k=-2$ and simplify the left-hand side using equation (25.11.12) in \cite{dlmf}.
\end{proof}
\begin{example}
The Derivative of the Riemann zeta function $\zeta'(3)$ and Apery's constant $\zeta(3)$.
\begin{multline}
\sum_{n,p \geq 0}\frac{1}{\Gamma (n+1) \Gamma (p+1)}B_p\left(\frac{p}{3}+1\right) E_n\left(\frac{n}{4}+1\right) (-1)^{n+p} 12^{n+p+2} (-3 n-4
   p+264)^{-n-p-4}\\
    \Gamma (n+p+2) \left((n (p-282)+260 p-73464) H_{n+p+1}+n (-p)+(n (-p)+282 n-260 p+73464)\right. \\ \left.
     \log
   \left(-\frac{n}{4}-\frac{p}{3}+22\right)+279 n-256 p+71556\right)\\
   =-\frac{\zeta
   '(3)}{2}+\frac{1}{4} \zeta (3) (\log (4)-1)-\frac{289853 \log (2)}{3456000}-\frac{259 \log
   (3)}{11664}-\frac{25523438671457 (\log (4)-1)}{85200014592000}\\
   -\frac{9 \log (5)}{2000}-\frac{\log
   (7)}{686}-\frac{\log (11)}{2662}.
\end{multline}
\end{example}
\begin{proof}
In this proof we will use equation (\ref{eq:dh2}) and set $a=22,\alpha=\gamma=1,\beta=\frac{1}{4},\delta=\frac{1}{3}$ and simplify using equation (1.6) in \cite{finch}.
\end{proof}
%
%
%
\section{Extended Generating Functions}
In this section we apply the methods of simultaneous equations and ordinary differential equations to derive extended forms involving the Bernoulli and Euler polynomials. The method involves finding the closed form solution after increasing the factorial in the denominator by 1. We first assign a general function to the right-hand side of the equation we wish to derive. Next we take the difference of these equations followed by taking the derivative of the equation we are solving for such that the left-hand side is the same as the difference of the equations. Next we equate the right-hand sides and solve the ordinary differential equation. \\\\
\subsection{Example 1: Euler's polynomial}
Starting with the initial formula given by;
\begin{equation}\label{ode:eq1}
\sum_{n=0}^{\infty}\frac{z^n e^{-n y z} E_n(x+n y)}{\Gamma(n+1)}=\frac{2 e^{x z}}{\left(e^z+1\right)
   (1-y z)}
\end{equation}
We wish to solve the formula given by;
\begin{equation}\label{ode:eq2}
\sum_{n=0}^{\infty}\frac{z^n e^{-n y z} E_n(x+n y)}{\Gamma(n+2)}=g(z)
\end{equation}
We next take the difference of equations (\ref{ode:eq1}) and (\ref{ode:eq2}) simplify to get;
\begin{equation}\label{ode:eq3}
\sum_{n=0}^{\infty}\frac{n z^n e^{-n y z} E_n(x+n y)}{\Gamma (n+2)}=-g(z)-\frac{2 e^{x
   z}}{\left(e^z+1\right) (y z-1)}
\end{equation}
Next we take the first partial derivative with respect to $z$ of equation (\ref{ode:eq2}) and multiply both sides by $\frac{z}{1-y z}$ such that the left-hand side is the same as equation (\ref{ode:eq3}) given by;
\begin{equation}\label{ode:eq4}
\sum_{n=0}^{\infty}\frac{n z^n e^{-n y z} E_n(x+n y)}{\Gamma (n+2)}=\frac{z g'(z)}{1-y
   z}
\end{equation}
Since the left-hand sides of equations (\ref{ode:eq3}) and (\ref{ode:eq4}) are the same we may equate the right-hand sides and derive the ordinary differential equation given by;
\begin{equation}
-\frac{z g'(z)}{1-y z}-g(z)-\frac{2 e^{x z}}{\left(e^z+1\right) (y
   z-1)}=0
\end{equation}
Solving the above ordinary differential equation with initial condition $g(0)=0$ and simplifying we get;
\begin{theorem}
For all $|Re(z)|<1,x,y\in\mathbb{C}$,
\begin{multline}
\sum_{n=0}^{\infty}\frac{z^n e^{-n y z} E_n(x+n y)}{\Gamma(n+2)}\\
=\frac{e^{y z}}{z (x-y)} \left(2 e^{z (x-y)}
   \, _2F_1\left(1,x-y;x-y+1;-e^z\right)\right. \\ \left.+(x-y) \left(\psi
   ^{(0)}\left(\frac{x-y}{2}\right)-\psi ^{(0)}\left(\frac{1}{2}
   (x-y+1)\right)\right)\right)
\end{multline}
\end{theorem}
where 
\begin{equation}
\int \frac{z^{a-1}}{z+1}dz=\frac{z^a \, _2F_1(1,a;a+1;-z)}{a},
\end{equation}
$_2F_1\left(a,b,c,z\right)$ is the hypergeometric function and $\psi ^{(0)}(z)$ is the zeroth derivative of the digamma function $\psi ^{(n)}(z)$.
\subsection{Example 2: Bernoulli's polynomial}
Repeating the above method we derive the generating function for Bernoulli's polynomial given by;
\begin{theorem}
For all $|Re(z)|<1,x,y\in\mathbb{C}$,
\begin{multline}
\sum_{n=0}^{\infty}\frac{z^n e^{-n y z} B_n(x+n y)}{\Gamma (n+2)}\\
=\frac{e^{y z}}{z}
   \left(\frac{e^{z (x-y)} \left(\left(x^2-2 x y+y^2\right) \Phi
   \left(e^z,2,x-y\right)+z (y-x) \,
   _2F_1\left(1,x-y;x-y+1;e^z\right)\right)}{(x-y)^2}\right. \\ \left.
   -\psi
   ^{(1)}(x-y)\right)
\end{multline}
\end{theorem}
\section{Discussion}
In this paper, we have presented a method for deriving a bilateral generating function involving the product the Bernoulli and Euler polynomials along with some interesting related forms using contour integration. We would like to apply this method to derive other generating functions in future work. The results presented were numerically verified for both real and imaginary and complex values of the parameters in the integrals using Mathematica by Wolfram.
\end{document}